\documentclass[11pt]{amsart}
\pdfoutput=1

\usepackage[T1]{fontenc}
\usepackage{lmodern}
\usepackage[english]{babel}
\usepackage{geometry}
\usepackage{amsfonts,amsmath,amssymb,amsthm,mathrsfs}
\usepackage{tikz}
\usetikzlibrary{calc}
\usepackage{caption}
\usepackage{subcaption}
\usepackage{xspace}
\usepackage{booktabs,tabularx}
\usepackage{dcolumn}
\usepackage{marvosym}
\usepackage{microtype}
\usepackage{enumerate}
\usepackage{mathtools}
\usepackage{booktabs}
\usepackage{url}
\usepackage{hyperref}
\usepackage{faktor}

\microtypesetup{tracking,kerning,spacing}
\microtypecontext{spacing=nonfrench}

 \hypersetup{colorlinks,linkcolor=blue,citecolor=blue}

\newcommand\CC{{\mathbb C}}

\newcommand\KK{{\mathbb K}}

\newcommand\RR{{\mathbb R}}
\newcommand\ZZ{{\mathbb Z}}

\newcommand\TT{{\mathbb T}}
\newcommand\cS{\mathcal{S}}

\newcommand\cM{\mathcal{M}}

\newcommand\cF{\mathcal{F}}
\newcommand\cP{\mathcal{P}}

\newcommand\cB{{\mathcal B}}

\newcommand\SetOf[2]{\left\{\left.#1\vphantom{#2}\ \right|\ #2\vphantom{#1}\right\}}
\newcommand\smallSetOf[2]{\{{#1}\,|\,{#2}\}}
\DeclareMathOperator{\conv}{conv}
\DeclareMathOperator{\val}{val}
\DeclareMathOperator{\rank}{rk}
\DeclareMathOperator{\size}{\#}
\DeclareMathOperator{\closure}{cl}
\DeclareMathOperator{\Dr}{Dr}
\DeclareMathOperator{\Gr}{TGr}
\newcommand{\M}{\mathcal{M}}
\DeclareMathOperator{\lin}{lin}
\newcommand\cU{\mathcal{U}}

\newcommand{\bigslant}[2]{{\raisebox{.2em}{$#1$}\left/\raisebox{-.2em}{$#2$}\right.}}

\theoremstyle{plain}
    \newtheorem{theorem}{Theorem}
    \newtheorem{corollary}[theorem]{Corollary}
    
    \newtheorem{proposition}[theorem]{Proposition}
\theoremstyle{definition}
    \newtheorem{remark}[theorem]{Remark}
    \newtheorem{example}[theorem]{Example}
    \newtheorem{definition}[theorem]{Definition}
    \newtheorem{conjecture}[theorem]{Conjecture}

    \newtheorem{question}{Question}

\title{On local Dressians of matroids}

\author{Jorge Alberto Olarte, Marta Panizzut  and Benjamin Schr\"oter}

\address{
	Institut f\"ur Mathematik,
	FU Berlin,
	Arnimallee 2, 14195 Berlin, Germany \\
	  E-mail: olarte@zedat.fu-berlin.de
	  }

\address{
  Institut f{\"u}r Mathematik,
  TU Berlin,
  Str.\ des 17. Juni 136, 10623 Berlin, Germany \\
  E-mail: panizzut@math.tu-berlin.de
}

\address{
	Department of Mathematical Sciences, Binghamton University, Binghamton, NY 13902, USA \\
	E-mail:schroeter@math.binghamton.edu
}

\begin{document}
\begin{abstract}
We study the fan structure of Dressians $\Dr(d,n)$ and local Dressians
$\Dr(\cM)$ for a given matroid $\cM$. In particular we show that the fan
structure on $\Dr(\cM)$  given by the three term Pl\"ucker relations
coincides with the structure as a subfan of the secondary fan of the
matroid polytope $P(\cM)$. As a corollary, we have that a matroid
subdivision is determined by its 3-dimensional skeleton. We also prove that the
Dressian of the sum of two matroids is isomorphic to the product of the
Dressians of the matroids. Finally we focus on indecomposable matroids.  We show that binary matroids are indecomposable, and we provide a non-binary indecomposable matroid as a counterexample for the converse. 
\end{abstract}
\maketitle
\section{Introduction}
\noindent 
Let $\KK_p$ be an algebraically closed field of characteristic $p$ with a non-trivial,  non-archimedean valuation. 
The {\sl tropical Grassmannian} $\Gr_p(d,n)$ is a rational polyhedral fan parametrizing \emph{realizable} $(d-1)$-dimensional tropical linear spaces in the tropical projective space $\mathbb{TP}^{n-1}$. These are contractible polyhedral complexes arising from the tropicalization of $(d-1)$-dimensional linear spaces in the projective space $\mathbb{P}^{n-1}_{\KK_p}$. The tropical Grassmannian is the tropical variety obtained from the tropicalization of the Pl\"ucker ideal $I_{d,n}$ generated by the algebraic relations among the $d\times d$-minors of a $d\times n$-matrix of indeterminates. It is the tropicalization of the Grassmannian Gr$(d,n)$. Its study has been initiated by Speyer and Sturmfels \cite{SpeyerSturmfels:2004}. In the paper the authors focus in particular on the tropical Grassmannian $\Gr_p(2,n)$, exhibiting a bijection with the space of phylogenetic trees with $n$ labeled leaves.

Herrmann, Jensen, Joswig and Sturmfels \cite{HerrmannJensenJoswigSturmfels:2009} studied the {\sl Dressian}  Dr$(d,n)$, an outer approximation of the tropical Grassmannian which parametrizes all $(d-1)$-dimensional tropical linear spaces in $\mathbb{TP}^{n-1}$. This is the tropical prevariety defined by the three  term Pl\"ucker relations among the generators of $I_{d,n}$. These relations induce the \emph{Pl\"ucker fan structure} on Dr$(d,n)$. From work of Speyer \cite{Speyer:2008} it follows that a point is in the Dressian if and only if it induces a matroid subdivision of the hypersimplex $\Delta(d,n)$. This endows Dr$(d,n)$ with a \emph{secondary fan structure} as subfan of the secondary fan of $\Delta(d,n)$. In  \cite{HerrmannJensenJoswigSturmfels:2009}, the authors proved that for $d=3$ the two fan structures coincide.

The Grassmannian Gr$(d,n)$ can be stratified in strata consisting of points with coordinates equal to zero if and only if they are not indexed by a basis of a matroid. As remarked in \cite{HerrmannJensenJoswigSturmfels:2009},  a similar stratification can be considered in the tropical setting. In particular, we can look at  the intersection of the Dressian Dr$(d,n)$ with each of the open faces of $\mathbb{TP}^{n-1}$. The intersection is not empty only if the face corresponds to a matroid of rank $d$ on $[n]$. This motivates the authors to give a similar definition for the  local Dressian Dr$(\cM)$ of a matroid $\cM$. In the article, they focused exclusively on this construction for the Pappus matroid. In our paper, we provide more examples and we analyze more deeply the properties of local Dressians. 

Local Dressians can also be endowed with two fan structures: one coming from the Pl\"ucker relations, one as a subfan of the secondary fan. Our main contribution is Theorem \ref{thm:structures} which states that the two fan structures coincide. The proof is based on a careful analysis of the subdivision induced by a point in the local Dressian on the $3$-dimensional skeleton of the matroid polytope. From our study it follows that a matroid subdivision is completely determined by its restriction to the $3$-skeleton. 

We then focus on local Dressians of disconnected matroids. We show that the local Dressian of the direct sum of two matroids is the product of their local Dressians.  Again, the key step in the proof is to look at the $3$-dimensional skeleton of the matroid polytope. 

Finally, we move our attention to indecomposable  matroids, i.e., matroids which do not admits matroid subdivisions of their matroid polytopes. The local Dressians of such matroids are linear spaces. We prove that binary matroids are indecomposable. Moreover, we give a counterexample for the converse, exhibiting a indecomposable non-binary matroid. 

Many questions related to the indecomposability of matroids arose during the work for this manuscript. We conclude with a short section collecting them, including a conjecture.

\section*{Acknowledgments}
\noindent
We are very grateful to Takayuki Hibi and Akiyoshi Tsuchiya for the kind hospitality and the organization  of  the \emph{Summer Workshop on Lattice Polytopes 2018}. Furthermore, we are indebted to Alex Fink for suggesting us to look at the matroid of Preposition \ref{prop:indecomp}. We also thank Michael Joswig and Felipe Rinc\'on for their useful comments. 
Research by the first and second author is supported by the Einstein Foundation Berlin. Research by third author is carried out in the framework of Matheon (Project \emph{MI6 - Geometry of Equilibria for Shortest Path}) supported by Einstein Foundation Berlin. The authors would also like to thank the Institute Mittag-Leffler for its hospitality during the
program \emph{Tropical Geometry, Amoebas and Polytopes}, where the collaboration  started. 

\section*{Notation} Before beginning, we fix some notation. Given $d \leq n$ non-negative integers, we define the sets $[n]:= \{1, 2, \dots, n\}$, and $\binom{[n]}{d}:= \{\textrm{subset $S$ of $[n]$ with $d$ elements}\}$. Moreover, given $S \subset [n]$ and $i,j \in [n]$, we use the notation $Sij:= S\cup \{i,j\}$.
Furthermore, we denote the all ones vector by $\bf 1$.

\section{Tropical Grassmannians and tropical linear spaces}
\noindent
We begin with some basics about tropical geometry following Maclagan and Sturmfels \cite{MaclaganSturmfels:2015}, focusing in particular on the definition of tropical Grassmannian. We work over the tropical semiring $\big(\mathbb{T} = \mathbb{R} \cup \{\infty\}, \oplus, \odot\big)$, where  arithmetic is  defined by $a \oplus b = \min \{a,b\}$ and $a \odot b = a+b$. 

Let $\KK_p$ be an algebraically closed field of carachteristic $p\ge 0$  with a non-trivial non-archimedean  valuation $\val: \KK_p \to \RR \cup \{\infty\}$. Examples are the field of Puiseaux series and their generalizations with real exponents; see Markwig \cite{Markwig:2010}. Given a polynomial $f \in \KK_p[x_1, \dots, x_n]$, 
  \[f(x_1, \dots, x_n) = \sum_{u = (u_1, \dots, u_n)} c_u x_1^{u_1} \cdot \cdots \cdot x_n^{u_n}, 
 \]
 its {\sl tropicalization $\text{trop}(f)$} is 
 \[
 \text{trop}(f)(x_1, \dots, x_n) = \bigoplus_{u =(u_1, \dots, u_n)} \val(c_u) \odot x_1^{\odot u_1} \odot \cdots \odot x_n^{\odot u_n}.
 \]
 The {\sl tropical hypersurface} trop$(V(f))$ is defined as the set of points $w \in  \RR^n$ such that the minimum in trop$(f)(w)$ is attained at least twice. Given an ideal $I \subseteq \KK_p[x_1, \dots, x_n]$, the {\sl tropical variety} trop$(V(I))$ is the intersection of the tropical hypersurfaces trop$(V(f))$, with $f \in I$.  A {\sl tropical prevariety} is the intersection of finitely many tropical hypersurfaces. Any tropical variety is a tropical prevariety as it is the intersection of the hypersurfaces of a tropical basis; see Hept and Theobald \cite{HeptTheobald:2009} and  \cite[Section 2.6]{MaclaganSturmfels:2015} for more details.
 
Any  vector $w \in \RR^n$ defines a partial term order on the polynomial ring $\KK_p[x_1, \ldots, x_n]$. Given an homogeneous ideal $I$ the set of initial ideals $\text{in}_w(I)$ endows $\RR^n$ with the structure of polyhedral complex called {\sl Gr\"obner  complex}. The following result, known as {\sl Fundamental Theorem of Tropical Algebraic Geometry}, gives the connection between algebraic and tropical varieties, and subcomplexes of the Gr\"obner complex. 
\begin{theorem}[{\cite[Theorem 3.1.3]{MaclaganSturmfels:2015}}] Let $I$ be an ideal in $\KK_p[x_1, \dots, x_n]$ and $V(I)$ its variety intersected with the torus $(\KK_p\setminus \{0\})^n$. The following sets coincide in $\RR^n$: 
\begin{enumerate}
\item the tropical variety trop$(V(I))$;
\item the closure in $\RR^n$ of set of vectors $w$ such that $\text{in}_w(I)$ does not contain a monomial. 
\item the closure in $\RR^n$ of the set $\{(\val(w_1), \dots, \val(w_n)) \, | \, w \in V(I)\}$. 
\end{enumerate} 
\end{theorem}

When working with homogeneous polynomials it makes sense to consider tropical hypersurfaces and varieties in the tropical torus $\RR^n/\RR {\bf 1} \cong \RR^{n-1}$ or in its compactification $\mathbb{TP}^{n-1} = (\mathbb{T}^n  \setminus \{(\infty, \dots, \infty)\}) / \mathbb{R} {\bf 1}$. We will adopt both interpretations in this paper, making sure  to specify which one we are considering. 

\medskip 

Now, let $R$ be the polynomial ring in $\binom{n}{d}$ variables \[\ZZ\big[p_{i_1 \, i_2 \, \dots i_d} \, | \, i_1 < i_2 < \cdots < i_n \big].\]
We consider the {\sl Pl\"ucker ideal} $I_{d,n}$ in $R\otimes \KK_p$ generated by the algebraic relations among the $d \times d$-minors of any  $d \times n$-matrix in any field. The Grassmannian $\text{Gr}(d,n)$ is the variety $V(I_{d,n})$.  The ideal $I_{d,n}$ is generated by quadrics. The {\sl tropical Grassmannian $\Gr_p(d,n)$} is the tropical variety $\text{trop}(V(I_{d,n}))$. It is a pure $d(n-d)$-dimensional rational polyhedral fan in $\RR^{\tbinom{n}{d}-1} \cong \RR^{\tbinom{n}{d}}/\RR\bf{1}$. 

The study of tropical Grassmannian was initiated by Speyer and Sturmfels \cite{SpeyerSturmfels:2004}. The authors focused on $\Gr_p(2,n)$ and the special case $\Gr_p(3,6)$. The fan structure and the homology of the  tropical Grassmannian $\Gr_p(3,7)$ is studied in \cite{HerrmannJensenJoswigSturmfels:2009}.

\begin{theorem}[{\cite[Theorem 3.4 and Corollary 4.5]{SpeyerSturmfels:2004}}] The tropical Grassmannian $\Gr_p(2,n)$ is characteristic-free and coincides with the space of phylogenetic trees with $n$ labeled leaves. 
\end{theorem}

Classically, the Grassmannian Gr$(d,n)$ is an example of a moduli space. It parametrizes the $d$-dimensional subspaces of a $n$-dimensional $\KK_p$-vector space. A $d$-dimensional linear subspace of $\KK_p^n$ can be represented by a full rank $d\times n$ matrix. The $d\times d$-minors form the Pl\"ucker vector which is a point on the Grassmannian. This surjective map from linear spaces to Pl\"ucker vectors is called the Stiefel map. 

Now, let $v$ be a vector in the tropical Grassmanian $\Gr_p(d,n)$. For each subset $S\in \binom{[n]}{d+1}$, consider the tropical linear polynomial
\begin{equation}\label{eq:linearspace}
	f_S(v) := \bigoplus\limits_{i\in S}v_{S\setminus i}\odot x_i
\end{equation}
and define $L_v$ as the intersection of the tropical hyperplanes defined by $f_S$, as $S$ varies over all elements in $\binom{[n]}{d +1}$. 
Then, $L_v$ is the tropicalization of a classical $d$-dimensional linear space. 

Speyer in \cite[Proposition 4.5.1]{Speyer:2008} showed that every tropicalization of a linear space arises this way. Indeed, any classical linear space has Pl\"ucker coordinates. Taking the valuation of each Pl\"ucker coordinate yields a vector $v\in \Gr_p(d,n)$ such that $L_v$ coincides with the tropicalization of the linear space. We take the above as definition of \emph{realizable} tropical linear space. The tropical Grassmannian is the moduli space of realizable tropical  linear spaces.  

\begin{theorem}[{\cite[Theorem 3.8]{SpeyerSturmfels:2004}}]
	The bijection between the classical Grassmannian $\text{Gr}(d,n)$ and the set of $d$-planes in $\KK^n$ induces a unique bijection $v \mapsto L_v$ between the tropical Grassmannian and the set of realizable tropical $d$-planes in $n$-space. 
\end{theorem}

Let $A\in \RR^{d\times n}$ be a tropical matrix. The \emph{tropical Stiefel map} $\pi : \TT^{d\times n}\dashrightarrow \Gr(d,n)$ sends the matrix $A$ to the vector of its tropical minors $\pi(A)$. More precisely, for any $B = \{b_1,\dots,b_d\}\in \binom{[n]}{d}$ we have
\[
\pi(A)_B = \bigoplus\limits_{\sigma\in \cS_d} \bigodot \limits_{i=1}^nA_{b_\sigma(i),i}
\]
The tropical linear spaces which lie  in the image of the tropical Stiefel map are called Stiefel tropical linear spaces. These linear spaces have been studied by Ric\'on \cite{Rincon:2013}, Herrmann et al.\cite{HerrmannJoswigSpeyer:2012}, and Fink and Rinc\'on \cite{FinkRincon:2015}. Notice that not all tropical linear spaces arise in this way, i.e., the tropical Stiefel map is not surjective.

\section{Matroids and Dressians} 
\noindent
In this section we introduce the main characters of this paper: matroids and Dressians. The former are classical objects in discrete mathematics. They are an abstraction of the concept of linear independence. Nakasawa and Whitney introduced them independently in the 1930s. There are many cryptomorphic ways to define a matroid. We will present just one definition  and focus  on their relation to polyhedral structures. We refer the interested reader to Oxley \cite{Oxley:2011} and White \cite{White:1986}. 
A {\sl matroid} $\mathcal{M}$ is a pair $(E, \mathcal{B})$ where $E$ is a finite set and $\mathcal{B}$ is a non-empty collection of subsets of $E$ satisfying the {\sl base exchange property}: whenever $B$ and $B'$ are in $\mathcal{B}$ and and $e \in B \setminus B'$, there exists and element $f \in B'\setminus B$ such that $(B \setminus \{e\}) \cup \{f\}$ is also in $\mathcal{B}$. The sets in $\mathcal{B}$ are called {\sl bases}.  Each basis $B \in \mathcal{B}$ has the same number $d$ of elements, called the {\sl rank of $\cM$}. Given a subset $A\subseteq E$, the \emph{rank} $r(A)$ of $A$ is $\max\limits_{B\in \cB}|B\cap A|$. A \emph{flat} $F$ of $\cM$ is a subset of $E$ such that for every $e\in E\setminus F$ we have $r(F\cup e) = r(F)+1$.

\begin{example} Given a matrix $A$ with entries in a field $\KK$, the pair $(E, \mathcal{B})$ consisting of the set $E$ of columns of $A$ and the collection of the maximally independent subsets of $E$ is a matroid $\cM[A]$.
\end{example}

\begin{example}\label{ex:graphical} Given a (finite) graph, the pair $(E, \mathcal{B})$ consisting of the set of edges $E$ and the collection of the maximally spanning subsets of $E$ is a matroid.
\end{example}

One of the fundamental questions regarding matroid is about their representabilty. A matroid is {\sl representable} over a field $\KK$ if it is isomorphic to a matroid $\cM[A]$ for a matrix $A$ with entries in $\KK$. A matroid representable over the finite field with two elements is called {\sl binary}. A {\sl ternary matroid} is one representable over the finite field with three elements.  
A matroid that can be obtained from a graph as described in Example~\ref{ex:graphical} is called \emph{graphical matroid}. Graphical matroids are \emph{regular}, i.e., representable over any field.
It is a recent result of Nelson \cite{Nelson:2018} that almost no matroid is representable.
The following example provides non-regular matroids.

\begin{example} Let $E=[n]$ and $\mathcal{B}$ the collection of subsets of $E$ with $d$ elements. The matroid $(E, \mathcal{B})$ is the  {\sl  uniform matroid $\mathcal{U}_{d,n}$}.  The uniform matroid $U_{2,n}$ is not representable over a field with less than $n-1$ elements. In particular, the matroid $U_{2,4}$ is not binary.
\end{example}

\begin{example} The \emph{Fano matroid} $\cF_7$ is an example of a binary matroid that is only representable over fields of characteristic two. It is represented by all seven non vanishing $0/1$-vectors of length three over a field of characteristic $2$.  
\end{example}

The following operations on matroids are derived from taking minors of matrices.
Let $M = (E, \cB)$ be a matroid and $e\in E$. The \emph{deletion} of $e$
from $\cM$, denoted $\cM\backslash e$, is the matroid $(E\setminus
\{e\},\{B\in\cB \mid e\notin B\})$.
The \emph{contraction} of $e$ from $\cM$, denoted $\cM/e$, is the matroid
$(E\setminus \{e\},\{B\setminus e \mid e\in B \in \cB\})$.
Any matroid that is the result of successive deletions and contractions of
$\cM$ is called a \emph{minor} of $\cM$.
The \emph{dual} of $\cM$, denoted $\cM^*$, is the matroid $(E, \{E\setminus
B \mid B\in \cB\})$. It is straightforward to verify that $(\cM^*)^* =
\cM$, and that $\cM^* \backslash e = (\cM/e)^*$.

\vspace{\baselineskip} 

We are most interested in the polyhedral point of view of defining and studying matroids. We fix $E = [n]$ as ground set. Let $e_1, e_2, \dots, e_n$ be the canonical basis of $\RR^n$. For a collection $\mathcal{S}$ of subsets of $E=[n]$, we define the polytope 
\[P_{\mathcal{S}} := \conv\SetOf {e_S}{S \in \mathcal{S}}, \]
where $e_S := \sum_{i \in S} e_i$. 
The {\sl $d$-th hypersimplex} in $\RR^n$ is the polytope 
\[\Delta(d,n) := P_{\binom{[n]}{d}}.\] 
A subset $\cM \subset \binom{[n]}{d}$ is a {\sl matroid of rank $d$ on $n$ elements} if the edges of $P_\cM$ are parallel to the edges of $\Delta(d,n)$, i.e., they are of the form $e_i -e_j$ for $i, j \in [n]$ distinct. The elements in $\cM$ are the {\sl bases} and $P_\cM$ is a {\sl matroid polytope}. The fact that this construction gives a matroid is a result of Edmonds \cite{Edmonds:1970}. See also Gelfand, Goresky, MacPherson and Serganova \cite{GGMS:1987}. 

In terms of the matroid polytope, we have that
\[ P_{\cM\backslash i} \cong P_{f(\cM)} \cap \{x_i = 0\}, \ \text{and} \ P_{\cM/i} \cong P_{\cM} \cap \{x_i = 1\}. \]
Moreover, $P_{\cM^*} = f(P_{\cM})$ where $f$ is the affine involution  that  sends $x_i$ to $1-x_i$ for each coordinate $i$. In  particular, the polytopes $P_{\cM}$ and $P_{\cM^*}$ are isomorphic.

\medskip

We now move to the definition of Dressians.  We will particularly highlight their connection with matroids and matroid polytopes. Among the quadric generators of the Pl\"uker ideal $I_{d,n}$ are the {\sl three term Pl\"ucker relations}
\begin{equation} \label{eq:pluckerrelations} p_{Sij} \, p_{Skl} - p_{Sik} \, p_{Sjl} + p_{Sil} \, p_{Sjk}, 
\end{equation}
where $S \in \binom{[n]}{d-2}$ and $i,j,k,l \in [n]\setminus S$ pairwise distinct. The {\sl Dressian} $\Dr(d,n)$ is the tropical prevariety  in $\RR^{\tbinom{n}{d}}/\RR{\bf 1} \cong \RR^{\tbinom{n}{d} -1}$  defined by the Pl\"ucker relations. This means that for a vector $w$ in the Dressian $\Dr(d,n)$  the minimum of 
\begin{equation} \label{eq:pluckervect}
w_{Sij} + w_{S lm}\;, \quad w_{Sil} + w_{Sjm}\;, \quad  w_{Sim} + w_{Sjl}
\end{equation}
is achieved at least twice, where $S \in \binom{[n]}{d-2}$ and $i,j,l,m \in [n] \setminus S$ pairwise distinct. The name Dressian was proposed by Herrmann et al. \cite{HerrmannJensenJoswigSturmfels:2009} in honor of Andreas Dress who discovered these relations by looking at valuated matroids. We call a point in the Dressian \emph{valuated matroid}. The three term Pl\"ucker relations endow Dr$(d,n)$ with the {\sl Pl\"ucker fan structure}. 
The Dressian can be also viewed as a subcomplex in the tropical projective space $\mathbb{TP}^{\tbinom{n}{d}-1}$.  We will do this in the next section, when we introduce local Dressians. 

It follows directly from the definition that the Dressian $\Dr(d,n)$ contains the tropical Grassmannian $\Gr_p(d,n)$ for any characteristic $p$. From the results in Maclagan--Sturmfels \cite{MaclaganSturmfels:2015}, it follows that $\Dr(2,n)$ = $\Gr_p(2,n)$ as fans and $\Dr(3,6) = \Gr_p(3,6)$ only as sets. The tropical Grassmannian $\Gr_p(3,7)$ depends on the characteristic $p$ of the field and $\Gr_p(3,7) \neq \Gr_2(3,7)$ for $p\neq 2$ due to the representability properties of the Fano matroid. This implies that the Dressian $\Dr(d,n)$ disagrees with the tropical Grassmannian $\Gr_p(d,n)$ for $d\geq 3$ and $n\geq 7$.
This fact is even reflected in their dimensions. The dimension of $\Dr(d,n)$ is of order $n^{d-1}$ for fixed $d$, while the dimension of $\Gr_p(d,n)$ grows linear in $n$, see \cite[Corollary~32]{JoswigSchroeter:2017}. 

As we said, the Dressian $\Dr(d,n)$ is the intersection of $\tbinom{n}{d+2}\tbinom{d+2}{4}$ tropical hypersurfaces coming from the three term Pl\"ucker relations. Note that these relations do not generate the Pl\"ucker Ideal $I_{d,n}$ for $n\geq d+3 \geq 6$, but they generate its image in the Laurent polynomials ring $\KK_p[p_{i_1  \dots i_d}^{\ \pm}  \, | \, i_1 < i_2 < \cdots < i_n]$, see \cite[Section 2]{HerrmannJensenJoswigSturmfels:2009}. The tropical variety defined by the ideal generated by the three term Pl\"ucker relations coincides with the tropical Grassmannian.  

The following proposition provides an upper bound for the number of elements in a tropical basis for the tropical Grassmannian $\Gr_p(d,n)$, i.e., the number of tropical hypersurfaces defining $\Gr_p(d,n)$. 

\begin{proposition}
	The tropical Grassmannian has a tropical basis of size:
	\[
		\binom{n}{d+1}\left( \binom{n}{d-1} -\binom{d+1}{2}  \right) + \binom{n}{d}-d(n-d) \leq  2^{2n+1} \enspace .
	\]
\end{proposition}
\begin{proof}
	This bound follows from Theorem 1 in  Hept--Theobald \cite{HeptTheobald:2009} as $I_{d,n}$ is a prime ideal of codimension $\tbinom{n}{d}-d(n-d)-1$ generated by $\tbinom{n}{d+1}\left( \tbinom{n}{d-1} -\tbinom{d+1}{2}  \right) $ polynomials. The generators can be read off from the prove of Theorem 14.6 in  Miller--Sturmfels \cite{MillerSturmfels:2005}.
\end{proof}
Note that this is a better estimation of the minimal size of a tropical basis than the one that can be derived from a general degree bound given in \cite[Example 9]{JoswigSchroeter:2018}.

\vspace{\baselineskip}

Our final  goal for this section is to explain the  relation of  the Dressian to a general concept in polyhedral geometry. Let $P$ be a polytope in $\RR^n$ with $m$ vertices and dimension $k$. Any vector $w \in \RR^m$ induces a regular subdivision of $P$. We think $w$ as a height function which lifts the vertex $v_i$ to the height $w_i$. By projecting the lower faces of the convex hull $\conv\smallSetOf{(v_i, w_i)}{v_i \, \textrm{vertex of} \ P} \subset \RR^{m+1}$ we get a subdivision of $P$. Vectors inducing the same subdivision form a relatively open cone. The collection of all these cones is the secondary fan of the polytope $P$. The lineality space is the largest linear space contained in each cone of the fan. The secondary fan has a $(k+1)$-dimensional lineality space that contains ${\bf 1}\in\RR^m$. In particular we may consider its image in $\RR^m / \RR{\bf 1}$. 

A subdivision of $\Delta(d,n)$ is a {\sl matroid subdivision} if each of its cell is a matroid polytope. Speyer proved a description of the Dressian in terms of matroid subdivisions. 
 \begin{theorem}[Proposition 2.2 in Speyer \cite{Speyer:2008}] \label{proposition:speyer} A vector $w \in \RR^{\binom{n}{d}}$ lies in the Dressian $\Dr(d,n)$ if and only if it induces a matroid subdivision of the hypersimplex $\Delta(d,n)$.  
 \end{theorem}
This description sees the Dressian $\Dr(d,n)$ as a subfan of the secondary fan of the hypersimplex $\Delta(d,n)$, and define the {\sl secondary fan structure} on $\Dr(d,n)$. Suppose that $d\geq 2$. For each $S \in \binom{[n]}{d}$ of cardinality $d-2$, and $i,j,l,m \in [n] \setminus S$, the points
\[
e_{Sij}, \, e_{Sil}, \,  e_{Sim}, \, e_{Sjl}, \,  e_{Sjm} \, \textrm{and} \,  e_{Slm} 
\]
define the vertices of a octahedron $O$, which is a $3$-dimensional face of the hypersimplex $\Delta(d,n)$.  A  point $w$ in  the Dressian $\Dr(d,n)$ induces a matroid subdivision of $\Delta(d,n)$. According to which of the three  inequalities and equations in (\ref{eq:pluckervect}) are satisfied, the subdivision induced by $w$ on $O$, determines  one of the three possible subdivision of the octahedron in two quadrilateral pyramids or the trivial subdivision. Herrmann et al. \cite{HerrmannJensenJoswigSturmfels:2009}  showed that for $d = 3$ the Pl\"ucker fan structure coincides with the secondary fan structure. In the next section we will prove that this holds in general. 

For any valuated matroid $v\in \Dr(d,n)$, we can define $L_v$ in the same way as we did for points in the Grassmanian in Section 2.
 We call such $L_v$ a tropical linear space. Note that, as there are valuated matroids that are not in the tropical Grassmanian, there are tropical linear spaces which are not realizable.
If $v\in\Gr_p(d,n)$ then all the faces of the subdivision induced by $v$ must be  polytopes of matroids representable in characteristic~$p$; see \cite[Example 4.5.4]{Speyer:2009} and \cite[Proposition 34]{JoswigSchroeter:2017}.

Given a valuated matroid $v\in \Dr(d,n)$, any point $x\in \RR^n$ defines a matroid $\cM_x$ by taking the face from the regular subdivision of $\Delta(d,n)$ that is minimized in the direction of $x$. In other words, the bases of $\cM_x$ are the sets $B\in\binom{[n]}{d}$ such that $v_B-\sum\limits_{i\in B}x_i$ is minimal. A loop in a matroid is an element which is contained in no bases.
The notation of a linear space in \eqref{eq:linearspace} generalizes to an arbitary Pl\"ucker vector $v\in\Dr(d,n)$  of its realizability. The following is a combinatorial description of such a tropical linear space.

\begin{proposition}[{\cite[Proposition 2.3]{Speyer:2008}}] The tropical linear space $L_v$ consists of exactly all points $x\in \RR^n$ such that $\cM_x$ has no loops.
\end{proposition}
For each loopless matroid $\cM$ whose polytope appears in the regular subdivision induced by $v$, there is a corresponding polyhedral cell in $L_v$ given by the closure of all the points $x$ such that $\cM_x = \cM$. 
This way the tropical linear space $L_v$ has the structure of a polyhedral complex. A cell is bounded in this complex if and only if $\cM$ has no coloops. The subcomplex of bounded cells is the tight span. Note that the polyhedral structure of $L_v$ is not unique. We will illustrate this by looking at the recession fan of $L_v$. 

Whenever $v\in \Dr(d,n)$ in the tropical projective space has only $0$ and $\infty$ as values, the tropical linear space $L_v$ coincides with the Bergman fan of the matroid $\cM$ whose bases are the coordinates where $v$ is $0$.
This is also the recession fan of $L_v$. The Bergman fan has a natural fan structure by the above arguments. We can equip the Bergman fan with a finer structure in the following way: for each flat $F$ of $\cM$, i.e., closed set, let $e_F = \sum_{i\in F}e_i$ be a ray in the Bergman fan.
And for every flag $F_1\subset\dots\subset F_k$ of flats we get a cone generated by $e_{F_1},\dots,e_{F_k}$.
When $v$ is any valuated matroid and $x\in L_v$, then $L_v$ is locally near $x$ the same as the Bergman fan of the matroid $M_x$. For further details see \cite{FeichtnerSturmfels:2005}, \cite{ArdilaKlivans:2006} and \cite[Chapter 4]{MaclaganSturmfels:2015}. Moreover, the introduction of Hampe \cite{Hampe:2015}  gives a broad overview about properties and developments of tropical linear spaces.

\section{Local Dressians} 
\noindent
Let $\cM = ([n], \mathcal{B})$ be a matroid on the set $[n]$ of rank $d$.  Let $b = |\mathcal{B}|$ be the number of bases of $\cM$. We consider the variety in $\mathbb{P}^{b-1}$ defined by the ideal $I_{\cM}$ obtained  from the Pl\"ucker ideal $I_{d,n}$ by setting all the variables  $p_B$ to zero, where $B$ is not a basis of $\cM$.  The variety $V(I_{\cM})$ is the realization space of the matroid $\cM$. It parametrizes all the $d$-dimensional linear subspaces of $\KK_p^n$ whose non-zero Pl\"ucker coordinates are the bases of $\cM$. In particular $V(I_{\cM}) = \emptyset$ if and only if the matroid is not representable over $\KK_p$. This gives a stratification of the Grassmannian Gr$(d,n)$, where the strata are defined as 
\[ \{ p \in \text{Gr}(d,n) \, | \, p_B = 0 \, \text{if and only if} \, B \not \in \mathcal{B}\}. \]
A variation of Mn\"ev's Universality theorem implies that the strata can be complicated as any algebraic variety.

\begin{remark} For the reader familiar with toric geometric, consider the  Grassmannian Gr$(d,n)$ over the complex numbers $\CC$. The algebraic torus $T = (\CC^*)^n$ acts on $\CC^n$ by $(t_1, \dots, t_n) \cdot (x_1, \dots, x_n) = (t_1 x_1, \dots, t_n x_n)$. The action is linear so it maps subspaces to subspaces. Therefore, it induces an action on the Grassmannian Gr$(d,n)$. Given a point $p \in Gr(d,n)$, the closure of the  orbit  $T \cdot p$ is a toric variety. Let $p$ be a point in the stratum defined by $\cM$. The image of $\overline{T\cdot p}$ through the moment map  is the matroid polytope $P_{\cM}$. For further reading we refer to \cite{GGMS:1987}. 
\end{remark}

We now look at a similar  local construction for the Dressian, i.e., we look at the Dressian $\Dr(\cM)$ of a matroid $\cM = ([n], \mathcal{B})$. This construction has been introduced by Herrmann et al. \cite{HerrmannJensenJoswigSturmfels:2009}. In the article the authors focus just on a single example where $\cM$ is the Pappus matroid of rank three on nine elements.

 The \emph{Dressian $\Dr(\cM)$ of a matroid $\cM$} is the tropical prevariety in $\RR^{b-1} \cong \RR^{b}/\RR{\bf 1} $ given by the set of quadrics obtained from the three term Pl\"ucker relations by setting the variables $p_B$ to zero, where $B$ is not a basis of $\cM$. The Dressian $\Dr(d,n)$ contains the Dressians of all matroids of rank $d$ on $n$ elements as subcomplexes at infinity.

Let us be more precise. From the coordinatewise logarithmic map $-\log$ we get a homoemorphism $\text{int}(\Delta_{n-1}) \to \RR^n / {\bf 1}\RR$. The tropical projective space $\mathbb{TP}^{n-1}$ is a compactification of the tropical torus $\RR^n / {\bf 1}\RR$, such that the pair $(\text{int}(\Delta_{n-1}), \Delta_{n-1})$ is homeomorphic to $( \RR^n / {\bf 1}\RR, \mathbb{TP}^{n-1})$. 

Given $Z$ a non-empty subset of $[n]$, we define the set 
\[ \mathbb{T}^n(Z) := \{  (w_1, w_2, \dots, w_n)  \in \mathbb{T}^n \, | \, w_i = \infty \, \text{if and only if} \, i \in Z\}.\]
The image of the  sets $\mathbb{T}^n(Z)$ through the quotient map give a stratification of the boundary of  $\mathbb{TP}^{n-1}$. See Section 5 of Joswig \cite{Joswig:book} for further details.

 The intersection of the closure of the Dressian $\Dr(d,n)$ in the tropical projective space $\mathbb{T}^{\tbinom{n}{d}-1}$ and the boundary stratum $\smallSetOf{w\in\mathbb{T}^{\tbinom{n}{d}-1}}{ w_S = \infty \text{ for } S\not\in\cM }$ agrees with the local Dressian $\Dr(\cM)$. Therefore the Dressian $\Dr(d,n)$ contains the Dressians of the matroid $\cM$ as subcomplex at infinity. 

\begin{remark} Our definition of the local Dressian $\Dr(\mathcal{U}_{d,n})$ of the uniform matroid agrees with the definition of the Dressian $\Dr(d,n)$ and is bases only on the three term Pl\"ucker relations.
The definition of the local Dressian given in \cite[Section 6]{HerrmannJensenJoswigSturmfels:2009} and  \cite[Definition 4.4.1]{MaclaganSturmfels:2015} takes all quadratic Pl\"ucker relations into account.
	These definitions agree, see \cite[Example 2.32]{BakerBowler:2017}. In the paper, the authors call the elements in our definition of the Dressian \emph{weak} matroids and the elements coming from all  quadratic Pl\"ucker relations  \emph{strong} matroids over the tropical hyperfield. 
\end{remark}

\medskip

The following statement follows from the definition of Dr$(\cM)$ and Theorem \ref{proposition:speyer}. 
\begin{corollary} \label{corollary:localSpeyer} A vector $w$ lies in the Dressian $\Dr(\cM)$ if and only if it induces a matroid subdivision of the matroid polytope  $P_{\mathcal{M}}$. 
\end{corollary} 

Therefore we have again two fan structures on the Dressian of a matroid $\cM$: one induced by the Pl\"ucker relations and one induced by the secondary fan.

\begin{theorem} \label{thm:structures}  Let $\cM$ be a matroid of rank $d$ on $n$ elements. The Pl\"ucker fan structure coincides with the secondary fan structure on $\Dr(\cM)$. 
\end{theorem}
\begin{proof} First, we take vectors $v$ and $w$ lying in the same cone of the secondary fan. They induce the same subdivision of the matroid polytope $P_\cM$,  in particular of the $3$-dimensional skeleton. Therefore $v$ and $w$ satisfy the same three term Pl\"ucker relations and lie in the same cone of the local Dressian equipped with the Pl\"ucker structure.

Now we focus on the viceversa. We take $v$ and $w$ lying in the same Pl\"ucker cone $C_v$. This means that they satisfy the same equations and inequalities coming from the  three term Pl\"ucker relations. By Corollary~\ref{corollary:localSpeyer}, they induce two matroid subdivisions $\Sigma_v$ and $\Sigma_w$ of $P_\cM$. We want to show that $\Sigma_v = \Sigma_w$. This will imply that $v,w$ are in the same secondary cone. By the fact that they satisfy the same Pl\"ucker relations, we know that $\Sigma_v|_{3-\text{skeleton}} =  \Sigma_w|_{3-\text{skeleton}}$ as the $3$-faces are either tetrahedra or octahedra. We pick $\sigma_v$ a maximal dimensional cell in $\Sigma_v$. We suppose that $\sigma_v$ is not in $\Sigma_w$. It means without loss of generality there are vertices $q_1$ and $q_k$ in the cell $\sigma_v$ such that $q_1$ and $q_k$ do not lie in a maximal dimensional cell of $\Sigma_w$. 
Let $q_1 \, q_2 \dots q_k$ be a path in the vertex-edge graph of the cell $\sigma_v$. We pick a cell $\sigma_w$ in $\Sigma_w$ that contains $q_1 \, \dots \, q_i$ for some $i\leq k$ and there is no cell in $\Sigma_w$ containing $q_1 \, \dots \, q_{i+1}$. 

	Now we have that $q_{i-1}$ and $q_{i+1}$ are at most of distance two. So we can use the base exchange axiom in the definition of a matroid to construct up to six points giving the unique face $F$ of $\sigma_v$ spanned by $q_{i-1}$ and $q_{i+1}$. The following situations may arise.
\begin{itemize}
	\item Either $F$ is a octahedron, then $F$ is subdivided in $\Sigma_w$ as $q_{i-1}$, $q_i$ are in $\sigma_w$ and $q_{i+1}$ is not. This is a contradiction to the fact that the subdivisions agree on the $3$-skeleton.

	\item If $F$ is a pyramid, it cannot be subdivided, therefore $F$ is a face of $\sigma_w$ and hence $q_{i+1}$ is a vertex of $\sigma_w$, and that contradicts our assumption.

	\item Similarly if $F$ is $2$-dimensional, i.e., a square or a triangle. 
\end{itemize}
Hence we conclude that both points $q_1$ and $q_k$ are in $\sigma_w$ and hence the subdivisions $\Sigma_v$ and $\Sigma_w$ agree.
\end{proof}

\begin{corollary}The Pl\"ucker fan structure on the Dressian $\Dr(d,n)$ as a fan in $\RR^{\binom{n}{d} -1}$ coincides with the secondary fan structure. 
\end{corollary}
\begin{proof} It is enough to consider the uniform matroid $\mathcal{U}_{d,n}$ in the previous statement. 
\end{proof}

\begin{corollary}\label{cor:uniDressCone} Let $d \geq 2$, and  $\Sigma$ and $\Sigma'$ be two matroid subdivisions of the hypersimplex $\Delta(d,n)$.
   If they induce the same subdivision on the $3$-skeleton, or equivalently on the octahedral faces of $\Delta(d,n)$, then $\Sigma$ and $\Sigma'$ coincide.
\end{corollary}

\begin{remark}
	The  above statement extends Proposition 4.3 and Theorem 4.4 by Herrmann et al. \cite{HerrmannJensenJoswigSturmfels:2009} and is the key in the algorithm in Section 6 of Herrmann et al. \cite{HerrmannJoswigSpeyer:2012} for computing (local) Dressians.
	Note that the abstract tree arrangements in Section 4 of Herrmann et al. \cite{HerrmannJensenJoswigSturmfels:2009} are a cover of the $3$-skeleton of the hypersimplex $\Delta(3,n)$ for $n\geq 6$ and the metric condition guarantees that the height functions agree on all three maps that contain a given vertex.  
\end{remark}

A \emph{connected component} $S$ of $\cM$ is a minimal non
empty subset with the property that $|S\cap B|$ is the same for every base
$B$ of $\cM$. Connected components of $\cM$ partition $[n]$. If $[n]$
is the only connected component, we say that $\cM$ is \emph{connected}.
We derive the following characterization of the lineality space which follows from the characterization of the dimension of a matroid polytope in terms of connected components by Edmonds \cite{Edmonds:1970} or Feichtner and Sturmfels \cite{FeichtnerSturmfels:2005}. Together with the fact that the secondary fan of a set of vertices has a lineality space of the same dimension as the  affine dimension of the set of vertices.

\begin{corollary}
	Let $b$ be the number of bases of a matroid $\cM$ on $n$ elements and with $c$ connected components. The lineality space of the Dressian $\Dr(\cM)$ in $\RR^b/\RR{\bf 1}$ is of dimension $dim P_\cM = n-c$.
\end{corollary}
\begin{proof} Adding a linear functions to the height function of a regular subdivision does not change the subdivision. Therefore the linealty space is the image of the  map $\RR^n\to\RR^b$ with $e_i\mapsto \sum\limits_{B\ni i} e_B$. 
\end{proof}

\begin{example}
The local Dressian of the uniform matroid $\mathcal{U}_{2,4}$ coinsides with the Dressian $\Dr(2,4)$. This is a $5$-dimensional pure balanced fan in $\RR^6/\RR{\bf 1}$ consisting of three maximal cells and a $3$-dimensional lineality space.  \end{example}

\begin{example}
	The local Dressian of the matroid $\mathcal{U}_{1,2}\oplus\mathcal{U}_{1,2}$ is a $2$-dimensional linear space in $\RR^4/\RR{\bf 1}$ spanned by $e_{13}+e_{14}$ and $e_{13}+e_{23}$. The corresponding matroid polytope $P_{\mathcal{U}_{1,2}}\times P_{\mathcal{U}_{1,2}}$ is a square, which has no finer matroidal subdivision.
\end{example}

Let us discuss two examples of local Dressians of non-regular connected ternary $(3,6)$-matroids.
These are matroids that are representable over the field with three elements, but are not representable over the field with two elements.
\begin{example}\label{ex:matroid1}
	Let $\cM$ be the matroid on 6 elements and rank 3 whose bases are $\binom{[6]}{3}\setminus\{123,145,356\}$, see Figure ~\ref{fig:example}. The polytope $P_\cM$ is full dimensional so the local Dressian $\Dr(\cM)$ has a lineality space of dimension $5$ in $\RR^{16}=\RR^{17}/\RR{\bf 1}$. The local Dressian is $6$-dimensional and consists of three maximal cones. These cones correspond to the vertex split with the hyperplane $x_2+x_4+x_6 = 0$ and two $3$-splits, i.e., a subdivision into three maximal cells that intersect in a common cell of  codimension $2$. The three maximal cells of one of those $3$-splits is illustrated in Figure~\ref{fig:subdiv1}. 
\end{example}
\begin{figure}[t]	
	\centering
	\includegraphics[width=0.3\textwidth]{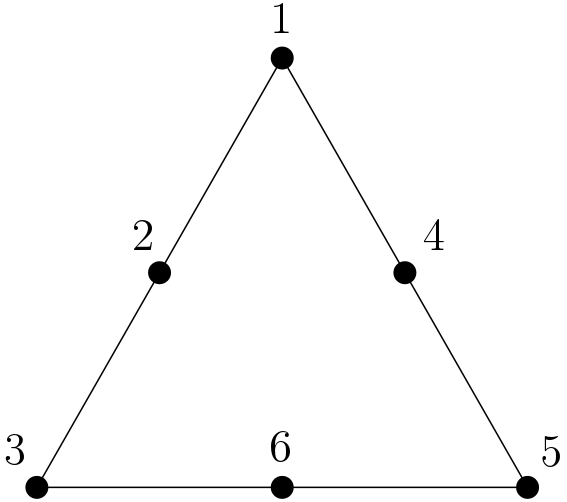}
	\caption{The terrnary matroid of Example~\ref{ex:matroid1}.}
\label{fig:example}
\end{figure}

\begin{figure}[t]	
\centering
\includegraphics[width=0.9\textwidth]{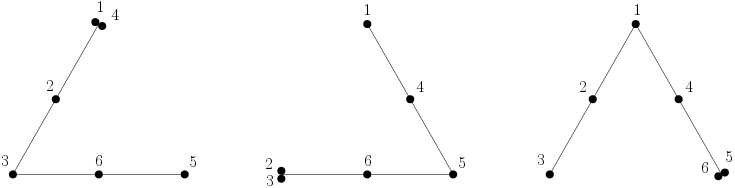}
\caption{The three matroids of one of the subdivisions of Example~\ref{ex:matroid1}. }
\label{fig:subdiv1}
\end{figure}

\begin{example}
	Let $\cM$ be the connected matroid given by the $14$ bases:
	\[
	135,\, 136,\, 145,\, 146,\, 156,\,
	235,\, 236,\, 245,\, 246,\, 256,\,
	345,\, 346,\, 356,\, 456 \enspace.
	\]
	The local Dressian $\Dr(\cM)$ consists of three maximal cones of dimension $6$ and a $5$-dimensional lineality space in $\RR^{13}$.
	In other words the polytope $P_\cM$ has four matroidal subdivisions.
	The trivial subdivision and three splits with respect to the hyperplanes $x_4+x_5+x_6 = 2$, $x_3+x_5+x_6=2$ or $x_3+x_4 = 1$.
\end{example}

\begin{remark} For any point $w$ in the local Dressian $\Dr(\cM)$ we can construct a tropical linear space $L_w$, by taking the intersection over $S \in \cM$ of the tropical hyperplanes defined by 
	\[
	f_S(w) = \bigoplus_{i\in S} w_{S\setminus i}\odot x_i \enspace .
	\]
	Further details on tropical linear spaces can be found in \cite[Section 4.4]{MaclaganSturmfels:2015}.
\end{remark}

\begin{proposition} Let $\cM$ and $\cM'$ be matroids such that $P_{\cM}$ is combinatorially isomorphic  to $P_{\cM'}$. Then, 
\[ Dr(\cM) \cong Dr(\cM'). \]
\end{proposition} 
\begin{proof}
	A matroid subdivision of the polytope $P_\cM$ does not impose new edges.
	The isomorphism between the polytopes $P_\cM$ and $P_{\cM'}$ induces a subdivision of $P_{\cM'}$ as images of cells. Moreover, this subdivision is matroidal as the $1$-cells are edges of $P_{\cM'}$.
	This subdivision is regular, as the map between $P_\cM$ and $P_{\cM'}$ is a concatenation of a coordinate permutation, an embedding and a reflection. This follows from the explicit description in Remark~\ref{rem:isom}. 
\end{proof}

\begin{remark}\label{rem:isom} It can be shown that the two matroid polytopes of $\cM$ and $\cM'$ are combinatorially isomorphic
	if and only if the matroids are isomorphic up to loops, coloops or dual connected components. This is part of the work by Pineda-Villavicencio and Schr\"oter \cite{PinedaSchroeter:2018}.
\end{remark}

The following statement deals with Dressians of disconnected matroids. Let $\cM_1$ and $\cM_2$ be matroids $(E_1, \mathcal{B}_1)$ and $(E_2, \mathcal{B}_2)$ with $E_1$ and $E_2$ disjoint. We define the \emph{direct sum} of $\cM_1$ and $\cM_2$ as 
\[\cM_1 \oplus \cM_2 = (E_1 \cup E_2, B_1 \cup B_2 \, \text{with} \, B_1 \in \mathcal{B}_1 \, \text{and} \, B_2 \in \mathcal{B}_2).
\] 
\begin{theorem}\label{thm:disconnected} Let $\cM_1$ and $\cM_2$ be matroids with disjoint element sets. Then
\[ \Dr(\cM_1 \oplus \cM_2) = \Dr(\cM_1) \times \Dr (\cM_2). \]
\end{theorem}

\begin{proof}
We have the map
\begin{align*}
\otimes: \enspace \Dr (\cM_1) \times \Dr(\cM_2) &\rightarrow \Dr(\cM_1 \oplus \cM_2)\\
(w,v) &\mapsto w\otimes v
\end{align*}
where $(w\otimes v)_{B_1\sqcup B_2} := w_{B_1}+v_{B_2}$ for any $B_1\in \cM_1$ and $B_2\in \cM_2$. To check that $w\otimes v$ satisfies the tropical Pl\"ucker relations notice the following:  any octahedron contained in $P_{\cM_1\oplus\cM_2}$ must be of the form $\{e_{B_1}\}\times O_2$,  with $B_1\in \cM_1$ and $O_2$ octahedron contained in $ P_{\cM_2}$, or  $O_1\times\{e_{B_2}\}$, with $B_2\in \cM_2$ and  $O_1$ octahedron contained in $P_{\cM_1}$. Then the Pl\"ucker relations follow from those of $\Dr(\cM_1)$ and $\Dr(\cM_2)$.
In particular, the cone where $w\otimes v$ lies is determined by the cones where $w$ and $v$ lie, so $\otimes$ maps cones into cones.

To construct the inverse of $\otimes$, we fix a basis $B_1\sqcup B_2\in \cM_1\oplus\cM_2$ and we define the map
\begin{align*}
\phi: \enspace \Dr(\cM_1 \oplus \cM_2) &\rightarrow \Dr(\cM_1) \times \text{Dr} (\cM_2)\\
w &\mapsto (\phi_1(w),\phi_2(w))
\end{align*}
where $\phi_1(w)_{A_1} := w_{A_1\sqcup B_2}$ and $\phi_2(w)_{A_2} := w_{B_1\sqcup A_2}$ for any $A_1\in \cM_1$ and any $A_2\in \cM_2$. It is straight forward to verify that the Pl\"ucker relations satisfied by $w$ imply that the projections $\phi_1(w)$ and $\phi_2(w)$ satisfy them as well. In particular, $\phi$ maps cones to cones.

Now we prove that $\phi$ is independent of the choice of basis $B_1\sqcup B_2$. We do this by contradiction. Suppose it is not, without loss of generality we can assume there exist $B_1 \sqcup B_2$ and $B_1\sqcup B_2'$, with  $B_2$ and $B_2'$ of distance 1 such that $\phi$ does not agree for these two choices. Clearly $\phi_2$ is the same for both choices, so we look at $\phi_1$. Let $A, A'\in \cM_1$ be bases at distance 1. We have that the points $e_{A\sqcup B_2},e_{A\sqcup B_2'},e_{A'\sqcup B_2'},e_{A'\sqcup B_2}$ form a square face of $P_{\cM_1\oplus\cM_2}$. This square can not be subdivided, so
\[
w_{A\sqcup B_2}-w_{A'\sqcup B_2}=w_{A\sqcup B_2'}-w_{A'\sqcup B_2'} \enspace .
\]
But this means that the difference of $\phi_1$ for $A$ and $A'$ is independent of the choice of $B_2$. By connectivity of the graph of $P_{\cM_1}$, we can conclude that $\phi_1$ is independent of the choice of $B_2$.

	We are left with proving that $\phi$ is the inverse of $\otimes$. First we check that for any $(v,w)\in \Dr(\cM_1) \times \Dr(\cM_2) $ we have that $\phi(w\otimes v) = (w,v)$. To see this, notice that $\phi_1(w\otimes v)_A = (w\otimes v)_{A\sqcup B_2} = w_A +v_{B_2}$ for any $A\in \cM_1$. But $v_{B_2}$ is a constant independent of $A$, so $\phi_1(w\otimes v) = w$ in the tropical torus. Analogously, we get that $\phi_2(w\otimes v) = v$.

Now we check the other direction, that is, for any $w\in \Dr(\cM_1 \oplus \cM_2)$ we have $w = \phi_1(w) \otimes \phi_2(w)$. Consider two bases of $(\cM_1 \oplus \cM_2)$ at distance 1. Without loss of generality let them be $A_1\sqcup A_2$ and $A_1\sqcup A_2'$. We have that 
\begin{align*}
 (\phi_1(w) \otimes \phi_2(w))_{A_1\sqcup A_2} -  (\phi_1(w) \otimes &\phi_2(w))_{A_1\sqcup A_2'} \\
 = & \phi_1(w)_{A_1}+\phi_2(w)_{A_2}-\phi_1(w)_{A_1}-\phi_2(w)_{A_1'}\\
 = & w_{A_1\sqcup B_2}+w_{B_1\sqcup A_2} - w_{A_1\sqcup B_2}-w_{B_1\sqcup A_2'}\\
= & w_{B_1\sqcup A_2}-w_{B_1\sqcup A_2'} \enspace .
\end{align*}

	We have already shown that $\phi$ is independent of the choice of $B_1$, so we may assume $B_1 =A_1$. Hence, the above equals $w_{A_1\sqcup A_2}-w_{A_1\sqcup A_2'}$. By connectivity of the graph of $P_{\cM_1\oplus\cM_2}$, we get $w = \phi_1(w) \otimes \phi_2(w)$ as we wanted. 

Therefore, the maps $\phi$ and $\otimes$ are bijective linear maps which send cones to cones, which implies $\Dr(\cM_1 \oplus \cM_2) =\Dr(\cM_1) \times \Dr(\cM_2)$.
\end{proof}

\begin{remark}
	The  statement above generalizes Theorem 4 by Chatelain and Ram\'irez \cite{ChatelainRamirez:2014} which deals with sequences of weakly compatible hyperplane splits. While the article by Joswig and Schr\"oter \cite{JoswigSchroeter:2017} provides the case of sequences of strongly compatible hyperplane splits and the matroid polytopes that occur in these matroid subdivisions.  We refer to Herrmann and Joswig~\cite{HerrmannJoswig:2008} for the definitions.
\end{remark}

Let $\cM$ be a matroid $(E, \mathcal{B})$. Two elements $e$ and $e'$ in $E$ are \emph{parallel} if $\rank(\{e,e'\}) = 1$. We denote this by $e \parallel e'$. Remark that this  implies that $\cM\backslash e = \cM \backslash e'$.

\begin{theorem}\label{thm:parallel} Let $\cM$ be a matroid and $e \parallel e'$ in $\cM$. Then
\[ \bigslant{\Dr(\cM)}{\lin\Dr(\cM)} \cong \bigslant{\Dr(\cM \setminus e')}{\lin\Dr(\cM \setminus e')} \]
and $\dim\lin\Dr(\cM) = \dim\lin\Dr(\cM \setminus e') +1$. 
\end{theorem}
\begin{proof}
	Clearly, $\cM$ contains the circuit $\{e,e'\}$. Hence, the number of connected components of $\cM$ is the same as the number of connected components of $\cM\setminus e'$. It  follows that $\dim\lin\Dr(\cM) = \dim\lin\Dr(\cM\setminus e')+1$.

	The projection $\Dr(\cM)\to\Dr(\cM\setminus e')$ that forgets the coordinates that correspond to bases that contain $e'$ is surjective.
	Our goal is to show that this projection is injective if we quoten by the lineality spaces.
	Let $w\in\Dr(\cM)$ and $B_e$ be a basis of $\cM$ that contains $e$ and $B_{e'} = B_e \setminus \{e\} \cup \{e'\}$. We may assume that $w_{B_e} = w_{B_{e'}}$ as the lineality space of $\Dr(\cM)$ contains $\sum_{B\ni e'} e_B$.
	Let $B'_{e'}$ be a basis of $\cM$ and $B'_e = B'_{e'} \setminus \{e'\} \cup \{e\}$ of distance $\size B'_e\setminus B_e = 1$. That is $e_{B_e}$, $e_{B'_e}$, $e_{B_{e'}}$, $e_{B'_{e'}}$, form a square in the vertex-edge graph of $P_\cM$.
	The set $B_e\cap B'_{e'}\cup\{e,e'\}$ is a non-basis of distance $1$ to those four bases. Therefore, the square is not subdivided by the regular subdivision induced by $w$.
	We conclude that $w_{B_e}+w_{B'_{e'}} = w_{B_{e'}}+w_{B'_e}$ and by our assumption $w_{B'_{e'}} = w_{B'_e}$.
	Iterating our argument shows that $w_{B} = w_{B\setminus\{e\}\cup \{e'\}}$ for any basis $B$ that contains $e$. As the basis exchange graph of a matroid is connected.
	Therefore, we derive that the projection is injective up to lineality and therefore the desired isomorphism.
\end{proof}

The combination of Theorem~\ref{thm:disconnected} and Theorem~\ref{thm:parallel} allows to deduce the local Dressian $\Dr(\cM)$ of an arbitary matroid $\cM$ from the simplifications of its connected componenets.  

\section{Indecomposable Matroids}
\noindent

In this section, we begin with  focusing on local Dressians of binary matroids, i.e., those matroids which are representable over the field with two elements. Recall that any matroid obtained from successive deletions and contractions form a matroid $\cM$ is a minor of $\cM$. The following is a useful characterization of binary matroids in terms of their minors.
\begin{proposition}[Tutte\cite{Tutte:1958}]\label{prop:binary_minors}
	A matroid is binary if and only if it has no minor isomorphic to the uniform matroid $\cU_{2,4}$.
\end{proposition}

From Proposition~\ref{prop:binary_minors} we derive a generalization of Theorem 3 by Chatelain and Ram\'irez \cite{ChatelainRamirez:2011} 
which states that a matroid polytope of a binary matroid can not be splited into two matroid polytopes.

\begin{definition}
A matroid is said to be indecomposable if and only if its polytope does not allow a non-trivial matroid subdivision.
\end{definition}

\begin{corollary} Let $\cM$ be a binary matroid. Then the local Dressian $\Dr(\cM)$ linear space. In particular, the matroid polytope $P_{\cM}$ is indecomposable.
\end{corollary}
\begin{proof}
	Let $\cM$ be a binary matroid and $P_\cM$ its matroid polytope.
	The $3$-skeleton of the polytope $P_\cM$ does not contain a octahedral face as such a face corresponds to a minor isomorphic to the uniform matroid $\cU_{2,4}$.
	From Corollary~\ref{cor:uniDressCone} we deduce that $P_\cM$ only has a trivial matroid subdivision.
	That is the Dressian is a linear space and $\cM$ is indecomposable.
\end{proof}

A matroid can be indecomposable even if its matroid polytope contains octahedral faces. Consider the simple matroid $\cP$ on $13$ elements and rank $3$ given by the ternary projective plane. This matroid is not binary and its matroid polytope has $117$ octahedral faces.
\begin{proposition}\label{prop:indecomp}
	The local Dressian $\Dr(\cP)$ is a $12$-dimensional linear space.
	In particular, the matroid $\cP$ is indecomposable.
\end{proposition}
\begin{proof}
	We will show the indecomposability by contradiction. We assume that $\cM$ is a proper connected submatroid of $\cP$. Being a submatroid means that every basis of $\cM$ is a basis of $\cP$.
	In this proof we will use the \emph{closure operators} of the matroids $\cM$ and $\cP$.
	Recall that the closure $\closure(S)$ of a  set $S$ is the maximal set that contains $S$ with $\rank(\closure(S)) = \rank(S)$. Here with maximal we mean that for every element $e\not\in \closure( S)$ we have  $\rank(e \cup \closure(S)) > \rank(S)$. 
	We denote by $\closure_{\cP}(S)$ the closure of $S$ in $\cP$ and by
	$\closure_{\cM}(S)$ the closure of $S$ in $\cM$.
	They satisfy $\closure_{\cP}(S)\subseteq\closure_{\cM}(S)$ whenever $\rank_{\cM}(S)=\rank_{\cP}(S)$.
	
	The proof consists of five steps:
	\begin{itemize}
	\item There are two parallel elements in $\cM$. 
	\item Two lines in $\cP$ collapse to a line in $\cM$.  
	\item A quadrilateral in $\cP$ collapses to a point in $\cM$. 
	\item Three concurrent lines in  $\cP$ collapse into a line in $\cM$.  
	\item Contradiction. 
	\end{itemize}
After each step, for improving the exposition, we reset the labeling. We make sure to clarify the new assigned labels. We do this in order to assure that there is no loss of generality. Keep in mind that for $\cM$ to be connected there is no line such that its complement is a single point. In particular there must be a least four points, i.e., four parallelism classes.

	Our first step is to show that $\cM$ contains a pair of parallel elements.
	Suppose that the set $123$ is a basis of $\cP$ but it is  dependent in $\cM$.
	Either $123$ contains a parallel pair or $\closure_{\cM}(123)$ is of rank $2$ as $\cM$ is loop free.
	In the latter case, let $4$ be not in the rank $2$ flat $\closure_{\cM}(123)$. 
	This implies that  the intersection of the lines  $\closure_{\cM}(14)\cap\closure_{\cM}(123)$ is of rank $1$ in $\cM$.
	As $2$ is not parallel to $3$ in $\cM,$ then $\closure_{\cP}(23)\subseteq \closure_{\cM}(23) = \closure_{\cM}(123)$ and, as $123$ is independent in $\cP$, there is an element $5$ in $\closure_{\cP}(23)\cap\closure_{\cP}(14)$. This means that  $5\in\closure_{\cM}(14)\cap\closure_{\cM}(123)$ and hence it is parallel to $1$ in $\cM$.

Suppose now that $1$ and $2$ are two parallel elements in $\cM$. Notice that there are at lest three elements not in $\closure_{\cM}(12)$. Moreover, $\closure_{\cP}(12)$ has four elements, at least two of which are in $\closure_{\cM}(12)$. Then there exists an element $3$ such that $3$ is not in $\closure_\cM(12) \cup \closure_\cP(12)$. Therefore, $\closure_{\cP}(13)$ and $\closure_{\cP}(23)$ are two different lines in $\cP$ which are contained in  $\closure_\cM(13) = \closure_\cM(23)$. 

	Suppose that the seven points on the two lines $1234$ and $1567$ in $\cP$ span a line in $\cM$. There must be at least two points $8$ and $9$ outside this line in the connected matroid $\cM$. Each of the three lines $\closure_\cP(28)$, $\closure_\cP(38)$ and $\closure_\cP(48)$ intersects the line $1567$ in a different point in the projective geometry $\cP$. This induces a bijection between $234$ and $567$ where elements are mapped to parallel elements in $\cM$. Similarly, a bijection can be constructed by considering the lines from $9$. These bijections do not agree and hence, there are at least four parallel elements in $\cM$ that span a quadrilateral in $\cP$.

Suppose that $1234$ is a quadrilateral in $\cP$ which collapses to a point in $\cM$. Let $5 \in\closure_\cP(12) \cap \closure_\cP(34)$, and $6 \in\closure_\cP(13) \cap \closure_\cP(24)$, and $7 \in\closure_\cP(14) \cap \closure_\cP(23)$. As $\cM$ is connected, there are at least three elements outside $\closure_{\cM}(1234)$. Suppose that these points are exactly  $5$, $6$ and $7$. Then $\closure_{\cP}(56) \cap \closure_{\cM}(1234) \not = \emptyset$ forcing $\closure_{\cM}(1234)$, $5$ and $6$ to be colinear in $\cM$, and $\cM$ disconnected. So there is another point $8$ outside $\closure_{\cM}(1234)$. In particular, three of the lines in $\cP$ passing through $8$ also pass through at least one point in the quadrilateral $1234$. Therefore they collapse in a single line in $\cM$. 

Suppose three concurrent lines passing through $1$ in $\cP$ collapse to a single line  in $\cM$. Let $S$ be the set of elements different from $1$ forming these three lines. As $\cM$ is connected there must be at least two elements outside $\closure_{\cM}(S)$. For each point, the lines passing through it and not $1$ induces a partition of $S$ in three subsets of size three, such that the elements in each subsets belong to the same parallelism class. The two partitions are transversal, therefore $S$ is in the same parallelism class. As the complement of $S$ is a line in $\cP$, then $\cM$ is disconnected and we obtain a contradiction. 
\end{proof}

\begin{remark}
We actually proved a stronger statement, namely that the matroid $\cP$ does not contain a proper connected submatroid. 
\end{remark}

We end this section by showing that a finest matroid subdivision of $\Dr(2,n)$ contains only indecomposable matroids.
\begin{proposition}\label{prop:subd_decomp}
	The cells of a finest matroid subdivision of $\Delta(2,n)$ correspond to binary matroids.
	In particular, they are indecomposable.
\end{proposition}
\begin{proof}
	Every matroid subdivisions of the hypersimplex $\Delta(2,n)$ is regular, as it is a sequence of compatible splits, i.e., subdivisions that divide $\Delta(2,n)$ into two maximal cells. See \cite{JoswigSchroeter:2017} and \cite{HerrmannJoswig:2008} for further details. 
	Every regular matroid subdivision of $\Delta(2,n)$ is representable in characteristic $0$ as $\Dr(2,n)=\Gr_0(2,n)$. Moreover, it is a finest subdivision if and only if it has $n-2$ maximal cells. 
	The Main Theorem in \cite{Speyer:2009} states that all cells in a subdivision of the hypersimplex $\Delta(d,n)$ are graphical and therefore binary whenever the subdivision is induced by a Pl\"ucker vector in $\Gr_0(d,n)$ and has $\tbinom{n-2}{d-1}$ maximal cells. This applies to the finest subdivisions of $\Delta(2,n)$ and hence  the maximal cells are matroid polytopes of binary matroids and indecomposable.
\end{proof}

\section{Open Questions}
\noindent  Several questions arise from the last section. We  end this article by stating them and in particular, by making a conjecture. 

A class of possible indecomposable matroids comes directly from the previous section. We conjecture the following generalization of Proposition \ref{prop:indecomp}:
\begin{conjecture}
All matroids that arise from projective spaces over finite fields are indecomposable. 
\end{conjecture}
Notice that a direct consequence would be examples of indecomposable matroids which are only representable over a particular characteristic.
While the direct sum of the ternary projective plane and the binary projective plane, i.e., fano matroid,  is an indecomposable matroid which is not representable over any field. We also  want to remark that as the ternary projective plane $\cP$ has a decomposable minor $\cU_{2,4}$,  a classification of indecomposable matroids can not relay on excluded minors.

Moreover, it would be interesting to find an efficient criterion to check indecomposability.  In Proposition~\ref{prop:indecomp} we used that there does not exist a connected submatroid. We wonder whether a submatroid is in general a certificate of decomposability.
\begin{question}
Does there exist two connected matroids $\cM$ and $\cM'$ such that $P_{\cM'}$ is
strictly contained in $P_\cM$ but no matroid subdivision of $P_\cM$ has
$P_{\cM'}$ as a cell?
\end{question}
Notice that when $\cM$ is a uniform matroid then the corank subdivision has $P_{\cM'}$ as a cell. 
But the corank function of $M'$ does not necessarily satisfy local Pl\"ucker relations.
\begin{example}
	Let $\cM$ be a matroid with bases $12$, $13$, $14$, $23$ and $24$ and
	$\cM'$ be the matroid with the two bases $12$ and $13$. Then the local corank lifting is $w = (0,0,1,1,1)$ and this vector is not in the local Dressian $\Dr(\M)$ as it subdivides the square pyramid $P_\cM$ into two tetrahedra.
\end{example}

Our last question is about finest matroid subdivisions of hypersimplicies and is derived from Proposition \ref{prop:subd_decomp}.
\begin{question}
	Are all cells in a finest matroid subdivision of a hypersimplex matroid polytopes of indecomposable matroids?
\end{question}

\end{document}